\documentclass[12pt]{article}

\pdfoutput=1 

\usepackage{eurosym}
\usepackage{amsthm,amssymb,amsmath}
\usepackage{epsfig,fge, mathtools}
\usepackage[dvipsnames]{xcolor}
\usepackage{pgf,tikz}
\usetikzlibrary{shapes.geometric}
\usepackage[latin1]{inputenc}
\usepackage[colorlinks=true, linkcolor=blue, citecolor=blue]{hyperref}
\usetikzlibrary{graphs}
\usepackage{mathdots}
\usepackage{graphicx}
\usepackage{epstopdf}
\usepackage{url}
\usepackage{cite}
\usetikzlibrary{decorations.markings}
\tikzstyle{vertex}=[circle, draw, inner sep=0pt, minimum size=6pt]

\newcommand{\decprob}[4]{
	\begin{center}
		\noindent\framebox{\begin{minipage}{#4\textwidth}
				#1\\
				\textbf{Instance:} #2\\ 
				\textbf{Question:} #3
		\end{minipage}}
	\end{center}
}

\newtheorem{them}{Theorem}

\newcounter{thmletter}

\newtheorem{definition}{Definition}
\newtheorem{theorem}{Theorem}
\newtheorem{corollary}{Corollary}
\newtheorem{lemma}{Lemma}
\newtheorem{remark}{Remark}
\newtheorem{proposition}{Proposition}
\newtheorem{observation}{Observation}

\tikzstyle{black_v}=[fill=black, draw=black, shape=circle, scale=0.85,minimum size=.5pt, very thin]
\tikzstyle{none}=[fill=none, draw=none, shape=circle, minimum size=1.5pt, very thin]

\begin{document}

\title{Total outer-independent coalition  in graphs}
\bigskip 
\author{Olga Buntina $^{1}$, Hamidreza Golmohammadi $^{2}$, \\ J. C. Valenzuela-Tripodoro$^3$}

\date{\today.}

\maketitle

\begin{center}
    $^{1}$ Novosibirsk State University, Pirogova Str. 2, Novosibirsk, Russia\\
    $^{2}$ Sobolev Institute of Mathematics, Ak. Koptyug av. 4, Novosibirsk, 630090, Russia\\
    $^{3}$ Universidad de C\'adiz, Algeciras Campus (Algeciras) 11202, Spain.
        \bigskip
    {\tt o.buntina@g.nsu.ru ~ h.golmohammadi@g.nsu.ru ~ jcarlos.valenzuela@uca.es } 
\end{center}

\def\nt{\noindent}

\maketitle

\begin{abstract}
A set $D$ of vertices graph $G$ is a total outer-independent dominating set (TOIDS) of $G$ 
if every vertex of $G$ is adjacent to at least one vertex in $D$,
and $V(G)\setminus D$ is an independent set of $G$. A TOI-coalition  in $G$ comprises two 
disjoint sets of vertices  $A$ and $B$ of $G$, neither of which is a TOIDS but whose union 
$A\cup B$ is a TOIDS of $G$. We say that the sets $A$ and $B$ form a TOI-coalition, and are 
TOI-coalition partners. A TOI-coalition  partition in $G$ is a vertex partition 
$\Psi =\{V_1, V_2,..., V_k\}$ in which every set forms a TOI-coalition with another set 
in $\Psi$. The TOI-coalition number $C_t^{oi}(G)$ is the maximum cardinality among all 
TOI-coalition partitions of $G$. In this work, the above-mentioned concepts are introduced 
and studied.  The existence of a TOI-coalition partition is investigated. Several sharp upper 
bounds on $C_t^{oi}(G)$ are established. Finally, the exact values of $C_t^{oi}(G)$ for some 
graph classes are obtained.
\end{abstract}

\noindent{\bf Keywords:}  
\medskip
\noindent{\bf AMS Subj.\ Class.:}  05C69

\section{Introduction and preliminaries }

Coalition in graphs is a new graph invariant, which was introduced by Haynes et al. \cite{A11}, 
and several papers have been devoted to this concept. The concept of coalition in graphs arises from
domination in graphs; it is not surprising that many variants of coalition have been
studied in relation to other domination invariants. For example, total
coalition, independent coalition, connected coalition, and total restrained coalition correspond
to total, independent, connected, and total restrained domination in graphs, respectively; 
see \cite{A1,A2,A3,Che2026, AA2}\medskip

Throughout this article, we only consider undirected simple graphs with vertex set $V(G)$ and
edge set $E(G)$. Given a vertex $v_i\in V(G)$, we denote by $N_{G}(v_i)$ and $N_{G}[v_i]$ the 
open neighborhood and the closed neighborhood of $v_i$ in $G$, respectively. The degree
of a vertex $u\in V(G)$ is $deg(u)=|N(u)|$. The minimum degree of a vertex in a graph $G$ is 
denoted $\delta(G)$, while the maximum degree of a vertex is denoted by $\Delta(G)$. The 
independence number of $G$, denoted $\alpha(G)$, is the maximum size of an independent set in $G$. The minimum number of vertices required to cover all edges of $G$ is called the vertex cover number and it is denoted by $\tau(G)$.
For graph theory notation and terminology, we generally follow \cite{W}.\medskip

A set of vertices in a graph $G$ is a dominating set if every vertex in $v\in V(G)\setminus S$ 
is adjacent to at least one vertex in $S$. Additionally, if every vertex in $G$ is adjacent to at 
least one vertex in $S$, then $S$ is a total dominating set (TDS). Domination and its variations 
in graphs have been thoroughly explored in the literature. For an in-depth introduction to domination 
in graphs, we refer to the book \cite{A15}. A total dominating set $S$ in $G$ is said to be a total 
outer-independent dominating set (TOIDS) if $V(G)\setminus S$ is independent. The minimum cardinality 
of a TOIDS in $G$ is denoted by $\gamma_t^{oi}(G)$. The concept of total outer-independent domination 
in graphs was initiated in \cite{TOID-upperbound}, and has subsequently been studied in 
\cite{A5, Cabrera-TCID-regular, Cabrera-TOID-2026, TOID-lowerbound, A17}. Substantial efforts have 
been directed towards partitioning the vertex set of a graph $G$ into subsets that possess a specific 
characteristic. A vertex partition of a graph $G$ into dominating sets is called the domatic partition 
of $G$. The domatic number of $G$, denoted by $d(G)$, is the maximum cardinality of a domatic partition 
of $G$. A survey on domination partitions in graphs can be found in \cite[Chapter 12]{A15}. Two 
disjoint sets $A, B \subseteq V(G)$ form a coalition if neither set is a dominating set in $G$ but 
their union is a dominating set. A coalition partition $\Psi$ of $G$ is a partition of $V(G)$ in which 
every set is either a singleton dominating set or forms a coalition with another set in $\Psi$. The 
coalition number $C(G)$ is the maximum cardinality taken over all coalition partitions of $G$. Given 
a coalition partition $\Psi=\{V_1,V_2 \ldots V_k\}$ of a graph $G$, the coalition graph $CG(G,\Psi)$ 
is the graph whose vertices correspond one-to-one with the sets of $\Psi$, and two vertices $V_i$ and 
$V_j$ are adjacent in $CG(G,\Psi)$ if and only if are coalition partners in $\Psi$. Coalitions in graphs, 
coalition partitions, and coalition graphs have been studied, for example, in \cite{A4,AA1,A12,A13,A14}.

In this paper, we study the properties of the TOI-coalition partition of a graph $G$. In Section 1, we 
establish notation and provide fundamental definitions employed in the paper. In Section 2, we investigate 
the existence of a TOI-coalition partition. In Section 3, we present upper bounds on the TOI-coalition 
number. In Section 4, we determine the exact value of $C_t^{oi}(G)$ for some well-known graphs $G$. Finally, we study the computational complexity of the decision problem of total outer-inde\-pen\-dent  coalition.

\begin{definition}[TOI-domatic partition]
\emph{A} total outer-independent domatic partition \emph{is a partition of $V(G)$ into TOIDS.} 
\emph{The} total outer-independent domatic number \emph{of $G$, denoted $d_t^{oi}(G)$, equals the 
maximum order of a total outer-independent domatic partition of $G$. A total outer-independent
domatic partition of cardinality $d_t^{oi}$ is called a $d_t^{oi}$-partition. }
\end{definition}

\begin{definition}[TOI-coalition]
\emph {Two nonempty disjoint vertex sets $X,Y\subseteq V(G)$} form a TOI-coalition , \emph{abbreviated} 
a  TOI-coalition, \emph{if neither $X$ nor $Y$ is a TOIDS of $G$, but $X\cup Y$ is a TOIDS of $G$.}
\end{definition}

\begin{definition}[TOI-coalition partition]
\emph{A partition $\Psi =\{V_1, V_2,..., V_k\}$ of the vertex set of a graph $G$ is called} a 
TOI-coalition partition \emph{if for every set $V_i\in\Psi$ there exists a set $V_j\in\Psi$, $j\neq i$, 
such that $V_i$ and $V_j$ form a TOI-coalition.} \emph{The} TOI-coalition number of $G$, \emph{denoted b
y $C_t^{oi}(G)$, equals the maximum order $k$ of a TOI-coalition partition of $G$.}
\end{definition}

\begin{definition}[TOI-coalition graph]
\emph{Associated with any TOI-coalition partition $\Psi=\{V_1,V_2,\dots,V_k\}$ in a graph $G$}, 
\emph{the} TOI-coalition graph $C_t^{oi}(G,\Psi)$ \emph{has the set of vertices $\Psi$, in which 
two vertices $V_i$ and $V_j$ are adjacent if they form a TOI-coalition in $G$}.
\end{definition}

\begin{remark}
Since a single vertex in a graph $G$ cannot be totally dominated, it follows that a singleton 
set is not a TOIDS in any TOI-coalition partition.
\end{remark}

\section{Existence conditions}

We begin with the existence problem for TOI-coalition partitions.

\begin{proposition}\label{Existence}
Any graph $G$ of minimum degree at least one has a TOI-coalition partition.
\end{proposition}

\begin{proof}
Let $\Psi=\{V_{1},V_2\ldots,V_k\}$ be a $d_t^{oi}(G)$-partition, and so $k=d_t^{oi}(G)$. We may assume 
that $V_{1},\ldots,V_{k-1}$ are minimal total outer-independent dominating sets. Otherwise, we substitute 
them with minimal total outer-independent dominating sets $V'_{1}\subseteq V_{1},\ldots, V_{k-1}'\subseteq 
V_{k-1}$ respectively, and substitute $V_{k}$ with $V_{k} \cup\big{(}\cup_{i=1}^{k-1}(V_{i}\setminus 
V_{i}')\big{)}$. Let $\{V_{i,1},V_{i,2}\}$ be any partition of $V_{i}$ for each $i\in[k-1]$. By the 
minimality of the TOIDS $V_i$, we note that neither set $V_{i,1}$ nor $V_{i,2}$ is a TOIDS of $G$. 
Nevertheless the set $V_i= V_{i,1} \cup V_{i,2}$ is a TOIDS of $G$.  Hence, we deduce that the sets 
$V_{i,1}$ and $V_{i,2}$ form a TOI-coalition  in $G$. If $V_{k}$ is a minimal TOIDS, then $\Psi'=
\{V_{i,1},V_{i,2}\}_{i=1}^{k}$ is a TOI-coalition  partition in $G$, in which $\{V_{k,1},V_{k,2}\}$ is 
any partition of $V_{k}$. Otherwise, we  substitute $V_{k}$ with a minimal TOIDS  $V_{k}'\subseteq V_{k}$ 
and set $V_{k}''=V_{k}\setminus V_{k}'$. Note that $V_{k}''$ is not a TOIDS  in $G$ as $\Psi$ is a 
$d_t^{oi}(G)$-partition. Let $\{V_{k,1}',V_{k,2}'\}$ be any partition of $V_{k}'$. If $V_{k}''$ forms 
a TOI-coalition  with $V_{k,1}'$ or $V_{k,2}'$, then $\{V_{i,1},V_{i,2}\}_{i=1}^{k-1}\cup 
\{V_{k,1}',V_{k,2}',V_{k}''\}$ would be a TOI-coalition partition. Thus, we suppose that neither 
$V_{k,1}'$ nor $V_{k,2}'$ forms a TOI-coalition  with $V_{k}''$. In such a situation, $V_{k,1}'\cup 
V_{k}''$ is a TOI-coalition partner of $V_{k,2}'$. Consequently, $\{V_{i,1},V_{i,2}\}_{i=1}^{k-1}\cup 
\{V_{k,1}'\cup V_{k}'',V_{k,2}'\}$ is a target partition.
\end{proof}

As an immediate result of Proposition \ref{Existence}, we have the following corollary.

\begin{corollary}
If a graph $G$ contains an isolated vertex, then $C_t^{oi}(G)=0$.
\end{corollary}

\section{Upper bounds}

In this section, we establish upper bounds on the TOI-coalition number of $G$.

\begin{proposition} \label{P2}
Let $G$ be a graph of order $n$ without isolated vertices. Then $C_t^{oi}(G)\le n-\gamma_t^{oi}(G)+2$. 
Moreover, this bound is sharp. 
\end{proposition}

\begin{proof}
Let $\Psi=\{V_1,V_2,\dots,V_k\}$ be a maximum TOI-coalition partition of $G$, where $k=C_t^{oi}(G)$. 
By definition, every set has a TOI-coalition partner. Without loss of generality, suppose that
$V_1$ and $V_2$ form a TOI-coalition. Then $V_1\cup V_2$ is a TOIDS, and therefore $\gamma_t^{oi}(G)
\le |V_1|+|V_2|$. Since the remaining sets are nonempty, it follows that $n=|V_1|+|V_2|+\cdots+|V_k| 
\ge \gamma_t^{oi}(G)+(k-2)$. Hence $k\le n-\gamma_t^{oi}(G)+2$. Since $k=C_t^{oi}(G)$, this completes 
the proof. The sharpness of this upper bound is illustrated in Theorem \ref{covering}. \end{proof}

The next result provides us with an upper bounds on $C_t^{oi}(G)$  in terms of the independence number of $G$.

\begin{proposition}\label{prop:upperalpha}
Let $G$ be a graph of order $n \geq 3$. Then  $C_{t}^{oi}(G) \leq \alpha(G) + 2.$
\end{proposition}

\begin{proof}
Let $\Psi = \{V_1, V_2, \ldots, V_k\}$ be a total outer-independent coalition
partition of $G$ with $|\Psi| = C_{toi}(G)$. By definition, for $V_1 \in \Psi$
there exists $V_j \in \Psi$, $V_j \neq V_1$, such that $V_1 \cup V_j$ is a
total outer-independent dominating set of $G$. In particular, the set
$V(G) \setminus (V_1 \cup V_j)$ is an independent set in $G$. Since $\Psi$ is a partition of $V(G)$, the remaining sets in $\Psi \setminus \{V_1, V_j\}$ are  pairwise disjoint nonempty subsets of $V(G) \setminus (V_1 \cup V_j)$. Since every subset of 
an independent set is independent, by selecting one vertex for each $V_i,$ for $i \in \{2,3, \ldots, k\}
\setminus\{j\}$, we obtain an independent set $\{v_i\} \subseteq V(G) \setminus (V_1 \cup V_j)$
of cardinality $k - 2$. Therefore, $k - 2 \leq \alpha(G),$
and the result follows. The sharpness of the bound is presented in Theorem \ref{covering}.
\end{proof}

Since $n = \tau(G) + \alpha(G)$ for any graph of order $n$, the following result is a straightforward 
consequence of Proposition \ref{prop:upperalpha}

\begin{corollary}\label{cor:uppertau}
Let $G$ be a graph of order $n \geq 3$. Then   $C_{t}^{oi}(G) \leq n-\tau(G) + 2.$
\end{corollary}

\begin{remark}
Let $\Psi$ be a TOI-coalition partition of a graph $G$, and let $X\in \Psi$. We denote the number 
of sets of $\Psi$ that form a TOI-coalition with $X$ by $C_{\Psi}(X)$. In the next result, we provide 
an upper bound on $C_{\Psi}(X)$.
\end{remark}

\begin{theorem} \label{t1}
Let $G$ be a graph with maximum degree $\Delta(G)$, and let $\Psi$ be a TOI-coalition partition of 
$G$. Then for each set $X\in \Psi$, $C_{\Psi}(X)\le \Delta(G)$.
\end{theorem}

\begin{proof}
Since $X\in \Psi$, the set $X$ is not a TOIDS. Two cases are possible.

\smallskip
\noindent
\textbf{Case 1.} $X$ is not a total dominating set. Then there exists a vertex $v\in V(G)$ such 
that $N(v)\cap X=\varnothing$. Let $Y\in \Psi$ be a TOI-coalition partner of $X$. Since $X\cup Y$ 
is a TOIDS, the vertex $v$ must have a neighbor in $X\cup Y$. Since $v$ has no neighbor in $X$, it 
follows that $N(v)\cap Y\neq\varnothing$. Thus every TOI-coalition partner of $X$ contains a vertex 
of $N(v)$. Since distinct partners are pairwise disjoint, there is a one-to-one correspondence between 
each partner and a distinct vertex of N(v). Therefore, $C_{\Psi}(X)\le |N(v)|=\deg(v)\le \Delta(G)$.

\smallskip
\noindent
\textbf{Case 2.} $X$ is a total dominating but not outer-independent. Since $V(G)\setminus X$ is not 
independent, in this case there exist adjacent vertices $u,v\in V(G)\setminus X$. Let $Y\in \Psi$ be 
a TOI-coalition partner of $X$. Then $X\cup Y$ is a TOIDS, implying that its complement
must be independent. Therefore, at least one of $u$ or $v$ must belong to $Y$, that is, $Y\cap 
\{u,v\}\neq\varnothing$. Thus, each TOI-coalition partner of $X$ includes at least one of the vertices 
$u$ or $v$. Since distinct partners are disjoint, there are at most two such partners. In this scenario, 
since $\Delta(G)\ge 2$, it follows that $C_{\Psi}(X)\le \Delta(G)$ once again.

\end{proof}

The following lemma will turn out to be useful in establishing the upper bound for the TOI-coalition 
number.

\begin{lemma} \label{l1}
Let $\Psi$ be a TOI-coalition partition of a graph $G$, let $X\in \Psi$, and let $Y\in \Psi$ be any
coalition partner of $X$. Then $|X\cup Y|\ge |V(G)|-\alpha(G)$. In particular, $|Y|\ge |V(G)|-\alpha(G)
-|X|$.
\end{lemma}

\begin{proof}
Since $X\cup Y$ is a TOIDS, the set $V(G)\setminus (X\cup Y)$ is independent. Hence $|V(G)\setminus 
(X\cup Y)|\le \alpha(G)$, implying that $|X\cup Y|\ge |V(G)|-\alpha(G)$. As $X\cap Y=\varnothing$, 
the second inequality follows immediately.
\end{proof}

\begin{proposition}\label{P1}
Let $\Psi$ be a TOI-coalition partition of a graph $G$, let $X\in \Psi$, and let $|X|<|V(G)|-\alpha(G)$. 
Then $C_{\Psi}(X)\le \left\lfloor \frac{|V(G)|-|X|}{\,|V(G)|-\alpha(G)-|X|\,}\right\rfloor$.
\end{proposition} 

\begin{proof} Let $Y_1,\dots,Y_m$ be all coalition partners of $X$, where $m=C_{\Psi}(X)$. By 
Lemma \ref{l1}, for every $i=1,\dots,m$, $|Y_i|\ge |V(G)|-\alpha(G)-|X|$. On the other hand, the 
sets $Y_1,\dots,Y_m$ are pairwise disjoint in $V(G)\setminus X$, so $|Y_1|+\cdots+|Y_m|\le |V(G)|-|X|$. 
Therefore $m\Bigl(|V(G)|-\alpha(G)-|X|\Bigr)\le |V(G)|-|X|$, yielding the desired result.
\end{proof}

By combining Proposition \ref{P1} and Theorem \ref{t1}, we infer the following result.

\begin{corollary} \label{C2}
Let $\Psi$ be a TOI-coalition partition of a graph $G$, and let $X\in \Psi$. If $|X|<|V(G)|-\alpha(G)$, 
then $C_{\Psi}(X)\le \min\left\{ \Delta(G), \left\lfloor \frac{|V(G)|-|X|}{\,|V(G)|-\alpha(G)-|X|\,}
\right\rfloor \right\}$.
\end{corollary}

\section{Exact values}

In this section, we determine exact values of $C_t^{oi}(G)$ for several specific families of simple
graphs $G$, including complete (bipartite) graphs, paths, and cycles. We first state the following 
observation.

\begin{observation}\label{ob1}
Let $D$ be a TOIDS of a graph $G$. Then no nonempty set $D_1\subseteq V(G)\setminus D$ is a TOIDS of $G$.
\end{observation}

We present next a key lemma, which is useful to derive further results.

\begin{lemma}\label{l2}
Let $\Psi$ be a maximum TOI-coalition partition of a graph $G$. Then there exists a set
$X\in \Psi$ such that $C_t^{oi}(G)=C_{\Psi}(X)+1$.
\end{lemma}

\begin{proof}
Consider the TOI-coalition graph $H_\Psi$. Since every set in $\Psi$ has a coalition partner, $H_\Psi$ 
has no isolated vertices. We claim that any two edges of $H_\Psi$ share a common endpoint. Suppose, to 
the contrary, that there exist at least two edges in $H_\Psi$ that do not have a common endpoint. We 
may thus assume that there exist four pairwise distinct sets $A$, $B$, $C$ and $D\in \Psi$ such that 
$A$ and $B$ form a TOI-coalition, and $C$ forms a TOI-coalition with $D$. Then both $A\cup B$ and 
$C\cup D$ are TOIDS of $G$. Since the sets of $\Psi$ are pairwise disjoint, this implies that the sets 
$A\cup B$ and $C\cup D$ are disjoint, contradicting Observation \ref{ob1}. Thus any two edges of $H_\Psi$ 
have a nonempty intersection. Since $H_\Psi$ has no isolated vertices, $H_\Psi$ must be isomorphic to 
either a star graph or a triangle. In both scenarios, there is a vertex that is adjacent to all the others. 
Let this vertex be associated with the set $X \in \Psi$. Hence, $C_{\Psi}(X)=|\Psi|-1$. Since $\Psi$ is 
maximum, it follows that $|\Psi|=C_t^{oi}(G)$, and therefore $C_t^{oi}(G)=C_{\Psi}(X)+1$.

\end{proof}

\begin{theorem}\label{covering}
The following equalities hold for any $n\geq3$. 
\begin{enumerate}
\item[\emph{(}i\emph{)}] $C_t^{oi}(K_n)=3$.
\item[\emph{(}ii\emph{)}] $C_t^{oi}(K_{r,s})=s+1$, where $r\le s$.
\item[\emph{(}iii\emph{)}] $C_t^{oi}(P_n)=
\begin{cases}
3, & n=3,\\
2, & n=4,\\
3, & n\ge 5.
\end{cases}$
\item[\emph{(}iv\emph{)}] $C_t^{oi}(C_n)=3$.
\end{enumerate}
\end{theorem}

\begin{proof} $(i)$  Let $n$ be an integer. First, consider the partition  $\Psi=\Bigl\{\{v_i\},\{v_j\},V(K_n) \setminus\{v_i, v_j\}\Bigr\}$ of $V(K_n)$. Since any vertex subset with cardinality $n-1$ of the complete graph  
is a TOIDS, we deduce that $\Psi$ is a TOI-coalition partition and hence $C^{oi}_t(K_n)\ge 3$. Finally, let $\Psi=\{X_1,\ldots,X_k\}$ be any TOI-coalition partition in $K_n$. Without loss of generality,
we may assume that $X_1,X_2$ form a TOI-coalition.  By taking into account the fact that $\alpha(K_n)=1$ and Lemma~\ref{l1}, we deduce that
$|X_1|+|X_2|\ge n-1$ and hence $|X_3 \cup \ldots \cup X_k|\le 1$, implying that $k\le 3$. Therefore, $C^{oi}_t(K_n)\leq 3$. This together with the lower bound leads to $C^{oi}_t(K_n)=3$. Note that the upper bound given in Proposition \ref{prop:upperalpha} is tight for $K_{n}$.

$(ii)$ Let $K_{r,s}$ be a complete bipartite graph with the partite sets $A$ and $B$ of cardinality 
$|A|=r$, $|B|= s$ and $s\geq r$. It is known that $\gamma_t^{oi}(K_{r,s})=\min\{r,s\}+1=r+1$ \cite{A17}. 
Therefore, by Proposition \ref{P2}, we have $C_t^{oi}(K_{r,s})\le (r+s)-(r+1)+2=s+1$. We show next that 
$C_t^{oi}(K_{r,s})\ge s+1$. Consider the partition $\Psi=\Bigl\{A,\{b_1\},\{b_2\},\dots,\{b_s\}\Bigr\}$, 
where $B=\{b_1,b_2,\dots,b_s\}$. The set $A$ cannot be a TOIDS as it cannot be totally dominated. Moreover, 
no singleton $\{b_i\}$ is a TOIDS. One can observe that, for each $i\in\{1,\dots,s\}$,  \( A \cup \{b_i\} \)
intersects both parts and contains the entire \( A \), and the set \( B \setminus \{b_i\} \) is independent. 
Thus $A\cup\{b_i\}$ is a TOIDS. Therefore, every singleton $\{b_i\}$ forms a TOI-coalition with $A$, and so 
$\Psi$ is a TOI-coalition partition. Hence, $C_t^{oi}(K_{r,s})\ge s+1$. By taking the upper bound and lower 
bound on $C_t^{oi}(K_{r,s})$ into account, we get $C_t^{oi}(K_{r,s})=s+1$. Note that the upper bound given 
in Proposition \ref{P2} is sharp for $K_{r,s}$.

$(iii)$ Let $\Psi$ be a TOI-coalition partition of $P_n$, and let $X \in \Psi$. If $|X| < |V(P_n)|-
\alpha(P_n)=\left\lfloor\frac{n}{2}\right\rfloor$, then Corollary~\ref{C2} yields $C_{\Psi}(X)\le 
\min\left\{ \Delta(P_n), \left\lfloor \frac{n-|X|}{\,\lfloor n/2\rfloor-|X|\,} \right\rfloor \right\}$. 
Note that $\Delta(P_n)=2$ and, on the other hand, $2\left\lfloor\frac{n}{2}\right\rfloor -n \le 0  
\le |X|, \ $ which implies that
$$ 2\left\lfloor\frac{n}{2}\right\rfloor  - 2|X|  \le n-|X|  \Longrightarrow 
	2   \le \frac{n-|X|}{\left\lfloor\frac{n}{2}\right\rfloor  - |X|}  $$
And hence, $C_{\Psi}(X)\le 2$.
Note that if $|X|\ge \left\lfloor\frac{n}{2}\right\rfloor$, then Corollary \ref{C2} is not valid. In 
such a situation, by Theorem \ref{t1}, we have $C_{\Psi}(X)\le 2$. By Lemma \ref{l2}, there exists a 
set $X \in \Psi$ such that $C_t^{oi}(P_n)=C_{\Psi}(X)+1$. Hence, $C_t^{oi}(P_n)\le 3$. We proceed 
further by proving that $C_t^{oi}(P_n)\ge 3$. We now let $n=3$.  In such a case, a partition consisting 
of three singletons is a desired partition. So, we have $C_t^{oi}(P_3)=3$. Further, we let
$n=4$. The partition $\Bigl\{\{v_1,v_2\},\{v_3,v_4\}\Bigr\}$ is a TOI-coalition for $P_4$. We can readily 
verify that $C_t^{oi}(P_4)\ne3$. Therefore $C_t^{oi}(P_4)=2$. Finally, let $n\ge 5$. Consider the 
partition $\Psi=\Bigl\{A=\{v_1\},B=\{v_3\},C=V(P_n)\setminus\{v_1,v_3\}\Bigr\}$. None of these sets is 
a TOIDS, while $A\cup C$ and $B\cup C$ are TOIDS. Therefore, $\{A,B,C\}$ is a TOI-coalition partition, 
and so $C_t^{oi}(P_n)\ge 3$. This, together with the upper bound, yields that $C_t^{oi}(P_n)=3$.

$(iv$) Consider the cycle $C_n$, where $n\ge 3$. Recall that $\alpha(C_n)=\left\lfloor\frac{n}{2}\right\rfloor$ and $\Delta(C_n)=2$. Using a similar idea as in the proof of Part $(iii)$, we deduce that $C_{\Psi}(X)\le 2$. Accordingly, by Lemma \ref{l2}, we have $C_t^{oi}(C_n)\le 3$. Now let $n=3$. In this case, the partition consists of three singletons and is a TOI-coalition partition. Hence, $C_t^{oi}(C_3)=3$. Next, assume that $n\ge 4$. Consider the partition $\Bigl\{V_1=\{v_1\}, V_2=\{v_3\}, V_3=V(C_n)\setminus\{v_1,v_3\}\Bigr\}$. None of the sets is a TOIDS, while $V_1\cup V_3$ and $V_2\cup V_3$ are TOIDS. In such a case, $\{V_1,V_2,V_3\}$ is a TOI-coalition partition. Hence, $C_t^{oi}(C_n)\ge 3$. This, together with the upper bound imply that $C_t^{oi}(C_n)=3$. 
\end{proof}

\section{Complexity}
The purpose of this section is to study the complexity of the following decision problem 
associated with total outer-inde\-pen\-dent coalition number of a graph.

\medskip
\decprob{ {TOTAL OUTER-INDEPENDENT COALITION PARTITION} \\ \mbox{} \hfill (TOIC-partition, for short)}
{A graph $G$, an integer $r$.}{Is $C_t^{oi}(G) \geq r$?}{0.75}
\bigskip

The TOIC-problem can be regarded as a vertex partition problem subject to both positive 
and negative constraints simultaneously. More precisely, it demands that the union of at 
least one pair of parts forms a TOID set, while no individual part is allowed to be. 
The combination of these conditions makes the problem at least as hard as its underlying 
subproblems; in particular, since deciding whether a graph admits a TOIDS of estimated size 
is already NP-hard, the additional partitioning constraints can only preserve or further 
increase the overall computational difficulty.

Nevertheless, we show that the problem becomes tractable when the input graph belongs to 
certain well-structured graph classes. Our approach relies on the framework developed by 
Courcelle et al.~\cite{courcelle}.

We begin by recalling the notion of a $k$-expression of a graph $G$. Given a graph $G = (V, E)$ 
and a set of labels $[k]$, a $k$-expression is a recursive construction of $G$ using the 
following four operations: $\eta_i(v):$ Introduces a new graph consisting of a single vertex 
$v$ labeled as $i \in [t]$; $G_1 \oplus G_2:$ Makes the disjoint union of two labeled graphs 
$G_1$ and $G_2$, without adding any edge between them; $\rho_{i,j}(G)$, with $i \neq j$: Adds 
an edge between every vertex of $G$ labeled $i$ and every vertex labeled $j$; and 
$\chi_{i \to j}(G)$, with $i \neq j$: Relabels every vertex with a label $i$ by assigning 
it label $j$ instead.
The clique-width of $G$, written $\text{cw}(G)$, is defined as the smallest integer $k$ for 
which $G$ admits a $k$-expression.

As an illustration, we give a $4$-expression for the house graph $H$ with vertex set 
$\{v_1, v_2, v_3, v_4, v_5\}$ and edge set $\{v_1v_2,\, v_1v_3,\, v_2v_4,\, v_3v_4,\, 
v_3v_5,\, v_4v_5\}$. The construction proceeds as follows.

\begin{enumerate}
    \item Disjoint union of $\{v_i\}$ with labels $i,$ for $i=1,2$:
    $G_1 = \eta_1(v_1) \oplus \eta_2(v_2).$
    \item Add the edge $v_1v_2$ by connecting labels $1$ and $2$, then 
    relabel $v_1$ with label $3$ to distinguish it from future vertices:
    $G_2 = \chi_{1 \to 3}\bigl(\rho_{1,2}(G_1)\bigr).$
    \item Introduce $v_3$ with label $1$ and add it to the current graph:
    $G_3 = G_2 \oplus \eta_1(v_3).$
    \item Add the edge $v_1v_3$, then relabel $v_1$ with label $4$:
    $G_4 = \chi_{3 \to 4}\bigl(\rho_{3,1}(G_3)\bigr).$
    \item Introduce $v_4$ with label $3$ and add it to the current graph:
    $G_5 = G_4 \oplus \eta_3(v_4).$
    \item Add edges $v_2v_4$ and $v_3v_4$ and relabel $v_4$ into $1$ so that 
    it shares a label with $v_3$:
    $$G_6 = \chi_{3 \to 1}\bigl(\rho_{1,3}\bigl(\rho_{2,3}(G_5)\bigr)\bigr).$$
    \item Introduce $v_5$ with label $3$ and add it to the current graph:
    $G_7 = G_6 \oplus \eta_3(v_5).$
    \item Add $v_3v_5$ and $v_4v_5$ simultaneously by connecting labels $1$ and $3$:
    $G_8 = \rho_{1,3}(G_7).$
\end{enumerate}

Since the resulting graph $G_8$ is precisely $H$, and the construction uses 
at most $4$ labels, so $\mathrm{cw}(H) \leq 4$.

The theorem below is a fundamental tool that makes it possible to solve a wide range of 
graph problems efficiently on graph classes of bounded clique-width. We point out that the 
result applies, in particular, to graphs of tree-width $t$: a tree-decomposition of width 
at most $t' = 2t+1$ can be obtained in linear time~\cite{K21}, and from it a 
$\left(3 \cdot 2^{t'-1}\right)$-expression can be constructed~\cite{CR05}.

\begin{them}\label{th:courcelle}{\emph{(Courcelle et al.~\cite{courcelle})}}
Let $\mathcal{C}$ be a class of graphs with clique-width bounded by a fixed constant $k$, 
that is, $\text{cw}(G) \le k$ for every $G \in \mathcal{C}$. Let $\Phi$ be a fixed sentence 
of Monadic Second-Order Logic $(\mathrm{MSOL}_1)$. If a $k$-expression of the input graph $G$ 
is provided together with the input, then the problem of deciding whether $G$ satisfies 
$\Phi$ is solvable in linear time on $k$ and the order.
\end{them}

Our goal is to formulate the decision version of the TOIC-partition problem within 
Monadic Second-Order Logic over graphs, known as $\mathrm{MSOL}_1$. This logic 
extends first-order logic by allowing quantification not only over individual 
vertices, but also over sets of vertices. This level of expressiveness is sufficient 
to describe a wide range of graph properties, including domination conditions and 
partition constraints, while remaining compatible with efficient algorithmic results 
on graph classes of bounded clique-width, such as Theorem~\ref{th:courcelle}.

To this end, we consider the relational structure $\langle G, \mathcal{R} \rangle$, where 
two vertices satisfy $\mathcal{R}(u, v)$ precisely when $uv$ is an edge in $E(G)$. 
We aim to express the TOIC-partition problem as a closed $\mathrm{MSOL}_1$ 
sentence $\Phi$ over this structure, so that $G$ satisfies $\Phi$ if and only if 
$C_t^{oi}(G) \geq r$.

\begin{theorem}\label{thm:CW}
The TOIC-partition problem can be solved in linear $O(f(k)\cdot n)$ time when restricted 
to graphs with clique-width at most $k$, if a $k$-expression is provided as part of 
the input.
\end{theorem}

\begin{proof}
    Let $G=(V,E)$ be a graph and let $r$ be a positive integer. To achieve our
aim, we use some auxiliary predicates defined as follows:
$$ N[u,v] \equiv \Big( \big( u=v \big) \lor \mathcal{R}(u,v) \Big)$$
where $N[u,v]$ holds if and only if $u$ belongs to the closed neighborhood 
of $v$ in $G$, that is, $u = v$ or $uv \in E(G)$. Also, 
$$ TDom(V_j) \equiv \forall v \bigg( \exists u 
    \Big( V_j(u) \land \mathcal{R}(u,v) \Big)\bigg) $$
which holds it and only if $V_j$ is a total dominating set in $G$.
And finally,
$$OInd(V_j) \equiv \forall u\, \forall v \bigg( \Big( \neg V_j(u) \land 
\neg V_j(v) \Big) \Rightarrow \neg \mathcal{R}(u,v) \bigg)$$
which captures the outer-independent condition for a vertex subset $V_j$.

We now let the formula $\Phi$ be defined as
\[ \Phi_{TOICP} : \exists \{V_1,\ldots,V_r\} \, \big( \Phi_P(V_1,\ldots,V_r) 
\wedge \Phi_{NTOID}(V_1,\ldots,V_r) \wedge \Phi_{TOIC}(V_1,\ldots,V_r) \big),
\]
where $\Phi_P$, $\Phi_{NLD}$ and $\Phi_{LDC}$ are subformulas that guarantee the 
sets $V_j$ constitute a TOIC-partition. 

The subformula $\Phi_P$ leads with the requirement that $V_1, \ldots, V_r$ form 
a partition of $V(G)$. It is given by
\[
\Phi_P(V_1,\ldots,V_r) \equiv \left( \bigwedge_{i=1}^r \bigwedge_{j=i+1}^r \neg\ \Big( \exists v 
\big( V_i(v) \land V_j(v) \big) \Big) \right) \land \left( \forall v 
\bigvee_{i=1}^r V_i(v) \right)
\]
The left part states that no vertex can belong to two distinct parts 
simultaneously, which guarantees that the sets are mutually disjoint. The 
right conjunct states that every vertex of $G$ must appear in at least one 
part, thus ensuring that the union of all parts covers $V(G)$ entirely.

The second subformula, $\Phi_{NTOID}$, will permit us to assure that no subset 
$V_j$ is a TOID-set. 
\[ \Phi_{NTOID}(V_1,\ldots,V_r) \equiv \bigwedge_{j=1}^r \neg \bigg( TDom(V_j) \land 
        OInd(V_j) \bigg), \]
which is true if and only if no subset $V_j$ is a TOID-set.

The third subformula, $\Phi_{TOIC}$, reflects the TOI coalition condition: for every 
part $V_j$ in the partition, there must exist some other part $V_i$ such that 
their union $V_i \cup V_j$ forms a TOIDS of $G$. This is expressed as
\[
\Phi_{TOIC}(V_1,\ldots,V_r) \equiv \bigwedge_{j=1}^{r} \bigvee_{\substack{i=1 \\ i \neq j}}^{r} 
\Big( TDom(V_i \cup V_j) \land OInd(V_i \cup V_j) \Big)
\]

At this point, the TOIC-partition problem has been fully encoded as a closed 
$\mathrm{MSOL}_1$ sentence 
$$\Phi_{TOICP} \equiv \exists V_1,\ldots,V_r \ \bigg(\Phi_P(V_1,\ldots,V_r) \land 
\Phi_{NTOID}(V_1,\ldots,V_r) \land \Phi_{TOIC}(V_1,\ldots,V_r) \bigg)$$ 
over the relational structure $\langle G, \mathcal{R} \rangle$. 
Hence, we can derive that the result holds, as an application of 
Theorem~\ref{th:courcelle}.
\end{proof}

\section{Concluding remarks}

We conclude by posing a few possible problems for future research.

\begin{enumerate}
    \item Establish upper and lower bounds for $C_t^{oi}(G)$ in terms of its maximum degree $\Delta(G)$ and minimum degree $\delta(G)$.

 \item Characterize isolate-free graphs $G$
 satisfying $C_t^{oi}(G)=n$.

\item Study the total outer-independent coalition number of trees.

\end{enumerate}

\noindent{\bf Acknowledgement.} 
The work of Hamidreza Golmohammadi was supported by the Mathematical Center in Akademgorodok 
under the agreement No. 075-15-2025-348 with the Ministry of Science and Higher Education of the Russian Federation.
The work of Juan Carlos Valenzuela-Tripodoro was partially supported by the Universidad de Cadiz under grant 2025-084/PU/PP-EST-INVEST-UCA/MV-EST 2025/203.

\end{document}